\documentclass[article, 10pt]{article}

\usepackage{amssymb}
\usepackage{graphics,graphicx,amsmath,latexsym,amssymb,amsthm,amssymb,amsfonts}
\newtheorem{thm}{Theorem}[section]

\newtheorem{defn}[thm]{Definition}
\newtheorem{rem}[thm]{Remark}

\newtheorem{exm}[thm]{Example}

\begin{document}

\title{Results on fuzzy soft topological spaces}
\author{J. MAHANTA*
        and P. K. DAS$^\perp$\\
Department of Mathematics\\
NERIST, Nirjuli\\
Arunachal Pradesh, 791 109, INDIA.\\ 
 *jm$\_$nerist@yahoo.in,  $^\perp$pkd$\_$ma@yahoo.com}

\maketitle
\thispagestyle{empty}
\begin{abstract}
B. Tanay et. al. \cite{S5} introduced and studied fuzzy soft topological spaces. Here we introduce fuzzy soft point and study the concept of neighborhood of a fuzzy soft point in a fuzzy soft topological space. We also study fuzzy soft closure and fuzzy soft interior. Separation axioms and connectedness are introduced and investigated for  fuzzy soft topological spaces.  
\end{abstract}

 
Fuzzy soft topological space, fuzzy soft closure, fuzzy soft connectedness, fuzzy soft separation . 


\section{Introduction and Preliminaries} 
Molodtsov \cite{S1} introduced the concept of soft sets in the year 1999. Fuzzy soft set was introduced by Maji et. al. \cite{S6}. Since then many researchers have been working for theoretical and practical development of this topic. B. Tanay et. al. introduced topological structure of fuzzy soft set in \cite{S5} and gave a introductory theoretical base to carry further study on this topic. S. Roy and T. K. Samanta also studied fuzzy soft topological space in \cite{S7}. 

This paper continues the study of Tanay et. al. to strengthen the theoretical pedestal of fuzzy soft topological spaces.

Here are some definitions and results required in the sequel.

Let $U$ be an initial universe, $E$ be the set of parameters, $\mathcal{P}(U)$ be the set of all subsets of $U$ and $\mathcal{FS}(U; E)$ be the family of all fuzzy soft sets over $U$ via parameters in $E$.

\begin{defn} \cite{S6}
Let $A \subset E$ and $\mathcal {F}(U)$ be the set of all fuzzy sets in $U$. Then the pair $(f, A)$ is called a fuzzy soft set over $U$, denoted by $f_A$, where $f: A \rightarrow \mathcal {F}(U)$ is a function.
\end{defn}
\begin{defn}
 Two fuzzy soft sets $f_A$ and $g_B$ are said to be disjoint if $f(a) \cap g(b) = \overset{\sim}{\Phi}, \forall a\in A, b\in B$.
\end{defn}
\begin{defn} \cite{S5}
Let $f_A $ be a fuzzy soft set, $\mathcal{FS}(f_A)$ be the set of all fuzzy soft subsets of $f_A $ and $\tau$ be a subfamily of $\mathcal{FS}(f_A)$. Then $\tau$ is called a fuzzy soft topology on $f_A $ if the following conditions are satisfied.
\begin{enumerate}
\item $\overset{\sim}{\Phi_A}, f_A$ belongs to $\tau$;
\item $h_A, g_B \in \tau \Rightarrow h_A\overset{\sim}{\bigcap} g_B \in \tau $;
\item $\{(h_A)_\lambda~|~ \lambda \in \Lambda\} \subset \tau \Rightarrow \overset{\sim}{\underset{\lambda \in \Lambda}{\bigcap}} (h_A)_\lambda \in \tau $.
\end{enumerate}
Then $(f_A, \tau)$ is called a fuzzy soft topological space. Members of $\tau$ are called fuzzy soft open sets and their complements are called fuzzy soft closed sets.
\end{defn}

\begin{defn} \cite{S5}
Let $(f_A, \tau)$ be a fuzzy soft topological space and $g_A \in \mathcal{FS}(f_A)$. Then the fuzzy soft topology $\tau_{g_A}= \{g_A \overset{\sim}{\bigcap} h_A~|~h_A \in \tau \}$ is called fuzzy soft subspace topology and $(g_A, \tau_{g_A})$ is called fuzzy soft subspace of $(f_A, \tau)$. 
\end{defn}

\begin{defn} \cite{S5}
Let $(f_A, \tau)$ be a fuzzy soft topological space and $h_A, g_B$ be fuzzy soft sets in $\mathcal{FS}(f_A)$ such that $g_B \overset{\sim}{\subset} h_A$. Then $g_B$ is called an interior fuzzy soft set of $h_A$ iff $h_A$ is a neighborhood of $g_B$. 

The union of all interior fuzzy soft sets of $g_A$ is called the interior of $g_A$ and is denoted by $g_A^0$.
\end{defn}

\begin{defn} \cite{S5}
Let $(f_A, \tau_1)$ and $(f_A, \tau_2)$ be two fuzzy soft topological spaces. If each $g_A \in \tau_1$ is in $\tau_2$, then $\tau_2$ is called fuzzy soft finer than $\tau_1$, or $\tau_1$ is called fuzzy soft coarser than $\tau_2$.
\end{defn}

\section{Fuzzy soft neighborhood, fuzzy soft closure and fuzzy soft interior}

In \cite{S5}, authors defined neighborhood of a fuzzy soft set but not for a point. Here we introduce and study fuzzy soft point and its fuzzy soft neighborhood. Further fuzzy soft interior and fuzzy soft closure of a fuzzy soft set in a fuzzy soft topological space are investigated.

\begin{defn}
A fuzzy soft set $g_A$ is said to be a fuzzy soft point, denoted by $\huge{e}_{g_{\scriptscriptstyle A}}$, if for the element $e \in A, g(e) \neq \overset{\sim}{\Phi}$ and $g(e^{'}) = \overset{\sim}{\Phi}, \forall e^{'} \in A-\{e\}$.
\end{defn}

\begin{defn}
The complement of a fuzzy soft point $e_{g_{\scriptscriptstyle A}}$ is a fuzzy soft point $(e_{g_{\scriptscriptstyle A}})^c$ such that $g_A^c(e)=1-g(e)$ and $g_A^c(e^{'})=\overset{\sim}{\Phi}~~\forall e^{'} \in A-\{e\}$.
\end{defn}

\begin{exm}
Let $U = \{h^1, h^2, h^3, h^4\}, A = \{e_1, e_2, e_3, e_4, e_5\} \subset E$, the set of parameters. Then $e_{g_{\scriptscriptstyle A}}= \{e_1=\{h^1_{0.1}, h^2_{0.9},  h^4_{0.4}\}\}$ is a fuzzy soft point whose complement is $(e_{g_{\scriptscriptstyle A}})^c= \{e_1=\{h^1_{0.9}, h^2_{0.1}, h^3_{1}, h^4_{0.6}\}\}$
\end{exm}

\begin{defn}
A fuzzy soft point $e_{g_{\scriptscriptstyle A}}$ is said to be in a fuzzy soft set $h_A$, denoted by $e_{g_{\scriptscriptstyle A} } \overset{\sim}{\in} h_A$ if for the element $e \in A, g(e) \leq h(e)$.
\end{defn}

\begin{thm}
Fuzzy soft points satisfy the following properties.
\begin{enumerate}
\item If a fuzzy soft point $e_{g_{\scriptscriptstyle A}} \overset{\sim}{\in}g_A$ then $e_{g_{\scriptscriptstyle A}} \overset{\sim}{\notin}g_A^c$;
\item $e_{g_{\scriptscriptstyle A}} \overset{\sim}{\in}g_A \nRightarrow e_{g_{\scriptscriptstyle A}}^c \overset{\sim}{\in}g_A^c$;
\item Union of all the fuzzy soft points of a fuzzy soft set is equal to the fuzzy soft set;
\item $e_{g_{\scriptscriptstyle A}} \in e_{h_{\scriptscriptstyle B}} \Leftrightarrow g(e)\leq h(e)$ and $A \subseteq B$;
\item $e_{g_{\scriptscriptstyle A}} \in \overset{\sim}{\bigcup} \{h_{\lambda \scriptscriptstyle B}~|~ \lambda \in \Lambda\} \Leftrightarrow ~\exists ~\lambda \in \Lambda$ such that $e_{g_{\scriptscriptstyle A}} \in h_{\lambda \scriptscriptstyle B}$;
\item $e_{g_{\scriptscriptstyle A}} \in \overset{\sim}{\bigcap} \{h_{\lambda \scriptscriptstyle B}~|~ \lambda \in \Lambda\} \Leftrightarrow ~\forall ~\lambda \in \Lambda ~ e_{g_{\scriptscriptstyle A}} \in h_{\lambda \scriptscriptstyle B}$.
\end{enumerate}
\end{thm}
Following is an example in favor of theorem 2.5.(ii).
\begin{exm}
Let $U = \{h^1, h^2\}, E = \{e_1, e_2\}$. Consider the fuzzy soft point $e_{g_{\scriptscriptstyle A}} = \{e_1=\{h^1_{0.1}, h^2_{0.2}\}\}$, which is contained in the fuzzy soft set $h_A = \{e_1=\{h^1_{0.1}, h^2_{0.9}\}, e_2 =\{h^1_{0.2}, h^2_{0.3}\}\}$. Then $h_A^c = \{e_1=\{h^1_{0.9}, h^2_{0.1}\}, e_2 =\{h^1_{0.8}, h^2_{0.7}\}\}$ does not contain $e_{g_{\scriptscriptstyle A}}^c = \{e_1=\{h^1_{0.9}, h^2_{0.8}\}\}$.
\end{exm}

\begin{defn}
 A fuzzy soft set $g_A$ in a fuzzy soft topological space $(f_A, \tau)$ is said to be a fuzzy soft neighborhood of a fuzzy soft point $e_{g_{\scriptscriptstyle A}}$ if  $\exists$ a fuzzy soft open set $h_A$ such that $e_{g_{\scriptscriptstyle A}} \overset{\sim}{\in} h_A \overset{\sim}{\subseteq} g_A$.
 \end{defn}
\begin{exm}
Consider fuzzy soft topological space $(f_A, \overset{\sim}{\tau})$ as defined in Example 3.2 of \cite{S5}.\\
Here $\{e_3 = \{h_{0.2}^1, h_{0.3}^2, h_{0.8}^3, h_{0.2}^4, h_{0.5}^5, h_{0.6}^6\}\}$ is a fuzzy soft neighborhood of the fuzzy soft point $\{e_3 = \{h_{0.1}^1, h_{0.2}^2, h_{0.8}^3, h_{0.2}^4, h_{0.5}^5, h_{0}^6\}\}$.
\end{exm} 
 
The family of all neighborhoods of $e_{g_{\scriptscriptstyle A}}$ is called its neighborhood system and is denoted by $\mathcal{N}_\tau (e_{g_{\scriptscriptstyle A}})$.

\begin{thm}
A fuzzy soft set in a fuzzy soft topological space is fuzzy soft open iff it is a fuzzy soft neighborhood of each of its fuzzy soft points.
\end{thm}

\begin{proof}
Let $(f_A, \tau)$ be a fuzzy soft topological space and $e_{g_{\scriptscriptstyle A}}$ be a fuzzy soft point in a fuzzy soft open set $g_{\scriptscriptstyle A}$. Then by definition, $g_{\scriptscriptstyle A}$ is a fuzzy soft neighborhood of $e_{g_{\scriptscriptstyle A}}$.

Conversely, let $g_{\scriptscriptstyle A}$ be a fuzzy soft set such that it is fuzzy soft neighborhood of each of its fuzzy soft points, say $e_{\lambda g_{\scriptscriptstyle A}}$. Then for each $\lambda \in \Lambda,~\exists$ a fuzzy soft open set $h_{\lambda \scriptscriptstyle A}$ such that $e_{\lambda g_{\scriptscriptstyle A}} \overset{\sim}{\subset} h_{\lambda \scriptscriptstyle A} \overset{\sim}{\subset} g_{\scriptscriptstyle A}$. Now $g_{\scriptscriptstyle A} = \overset{\sim}{\underset{\lambda \in \Lambda}{\bigcup}} e_{\lambda g_{\scriptscriptstyle A}} \Rightarrow g_{\scriptscriptstyle A} = \overset{\sim}{\underset{\lambda \in \Lambda}{\bigcup}} h_{\lambda \scriptscriptstyle A} \Rightarrow g_{\scriptscriptstyle A}$ is fuzzy soft open being the union of arbitrary family of fuzzy soft open sets. 

\end{proof}

\begin{thm}
The neighborhood system of $\mathcal{N}_\tau (e_{g_{\scriptscriptstyle A}})$ in a soft topological space $(f_A, \tau)$ satisfies the following properties:
\begin{enumerate}
\item If $g_A \in \mathcal{N}_\tau (e_{g_{\scriptscriptstyle A}})$, then $e_{g_{\scriptscriptstyle A}} \overset{\sim}{\in} g_A$;
\item A fuzzy soft superset of a fuzzy soft neighborhood of a fuzzy soft point is also a fuzzy soft neighborhood of the point;
\item Intersection of two fuzzy soft neighborhoods of a fuzzy soft point is again a fuzzy soft neighborhood;
\item $k_A \in \mathcal{N}_\tau (e_{g_{\scriptscriptstyle A}}) \Rightarrow ~\exists~ h_A \in \mathcal{N}_\tau (e_{g_{\scriptscriptstyle A}})$ such that $h_A \overset{\sim}{\subseteq} k_A$ and $ h_A \in \mathcal{N}_\tau (e_{h_ {\scriptscriptstyle A}})$.
\end{enumerate}
\end{thm}

\begin{proof}
\begin{enumerate}
\item If $g_A \in \mathcal{N}_\tau (e_{g_A})$, then $\exists$ a fuzzy soft open set $h_A$ such that $e_{g_A} \overset{\sim}{\in} h_A \overset{\sim}{\subseteq} g_A \Rightarrow e_{g_A} \overset{\sim}{\in} g_A$.
\item Let $g_A \in \mathcal{N}_\tau (e_{g_A}) \Rightarrow \exists$ a fuzzy soft open set $h_A$ such that $e_{g_A} \overset{\sim}{\in} h_A \overset{\sim}{\subseteq} g_A$  and $g_A \overset{\sim}{\subseteq} k_A \Rightarrow e_{g_A} \overset{\sim}{\in} h_A \overset{\sim}{\subseteq} k_A \Rightarrow k_A \in \mathcal{N}_\tau (e_{g_A})$.
\item Let $g_A, k_A \in \mathcal{N}_\tau (e_{g_A})$ then there exists fuzzy soft open sets $h_A$ and $s_A$ such that $e_{g_A} \overset{\sim}{\in} h_A \overset{\sim}{\subseteq} g_A$ and $e_{g_A} \overset{\sim}{\in} s_A \overset{\sim}{\subseteq} k_A \Rightarrow e_{g_A} \overset{\sim}{\in} h_A \overset{\sim}{\bigcap} s_A \overset{\sim}{\subseteq} g_A \overset{\sim}{\bigcap} k_A$. Now $h_A \overset{\sim}{\bigcap} s_A$ is fuzzy soft open and hence $g_A \overset{\sim}{\bigcap} k_A \in \mathcal{N}_\tau (e_{g_A})$.
\item $k_A \in \mathcal{N}_\tau (e_{g_A}) \Rightarrow ~\exists$ a fuzzy soft open set $s_A$ such that $e_{g_A} \overset{\sim}{\in} s_A \overset{\sim}{\subseteq} k_A$. By definition $s_A$ is a fuzzy soft neighborhood of each of its points, so $s_A \in \mathcal{N}_\tau (e_{s_A})$. 
\end{enumerate}
\end{proof}

\begin{defn}
Let $(f_A, \tau)$ be a fuzzy soft topological space and $g_A$ be a fuzzy soft set.
\begin{enumerate}
\item The fuzzy soft closure of $g_A$ is a fuzzy soft set\\ 
$fsclg_A= \overset{\sim }{\bigcap} \{ h_B~|~ g_A \overset{\sim }{\subseteq} h_B$ and $ h_B$ is fuzzy soft closed set$\}$;
\item The fuzzy soft interior of $g_A$ is a fuzzy soft set\\ 
$fsintg_A= \overset{\sim }{\bigcup} \{h_B~|~ h_B \overset{\sim }{\subseteq} g_A$ and $ h_B$ is fuzzy soft open set$\}$.
\end{enumerate}
\end{defn}

In \cite{S5}, authors defined fuzzy soft interior of a fuzzy soft set. But it is clear that both the definitions are equivalent.  

\begin{thm}
A fuzzy soft set $g_A$ is fuzzy soft closed iff $fsclg_A=g_A$.
\end{thm}
\begin{thm}
Let $(f_A, \tau)$ be a fuzzy soft topological space and $g_A, h_A$ be fuzzy soft sets. Then
\begin{enumerate}
\item $(fsclg_A)^c = fsintg_A^c$;
\item $(fsintg_A)^c = fsclg_A^c$;
\item $g_A \overset{\sim }{\subseteq} h_A \Rightarrow fsclg_A \overset{\sim }{\subseteq} fsclh_A$;
\item $g_A \overset{\sim }{\subseteq} h_A \Rightarrow fsintg_A \overset{\sim }{\subseteq} fsinth_A$;
\item $fscl(fsclg_A) = fsclg_A$;
\item $fsint(fsintg_A) = fsintg_A$;
\item $fscl \overset{\sim}{\Phi _A} = \overset{\sim}{\Phi _A}$ and $fscl f_A = f_A$;
\item $fsint \overset{\sim}{\Phi _A} = \overset{\sim}{\Phi _A}$ and $fsint f_A = f_A$;
\item $fscl(g_A \overset{\sim }{\cup} h_A) = fsclg_A \overset{\sim }{\cup}  fsclh_A$;
\item $fsint(g_A \overset{\sim }{\cap} h_A) = fsintf_A \overset{\sim }{\cap}  fsinth_A$;
\item $fscl(g_A \overset{\sim }{\cap} h_A) \overset{\sim }{\subset} fsclg_A \overset{\sim }{\cap}  fsclh_A$;
\item $fsint(g_A \overset{\sim }{\cup} h_A) \overset{\sim }{\subset} fsintg_A \overset{\sim }{\cup}  fsinth_A$;

\end{enumerate}
\end{thm}

\begin{proof}
Straightforward. 
\end{proof}

\begin{thm}
The fuzzy soft set $h_A$ is fuzzy soft closed in a subspace $(g_A, \tau_{g_A})$ of $(f_A, \tau)$ iff $h_A= k_A \overset{\sim}{\bigcap}g_A$ for some fuzzy soft closed set $k_A$ in $f_A$.
\end{thm}

\begin{thm}
The fuzzy soft closure of a fuzzy soft set $h_A$ in a subspace $(g_A, \tau_{g_A})$ of $(f_A, \tau)$ equals $fscl(h_A) \overset{\sim}{\bigcap}g_A$.
\end{thm}

\begin{proof}
We know $fsclh_A$ is a fuzzy soft closed set in $f_A \Rightarrow fsclh_A \overset{\sim}{\bigcap}g_A$ is fuzzy soft closed set in $g_A$. Now $h_A \overset{\sim}{\subset} fsclh_A \overset{\sim}{\bigcap}g_A$ and fuzzy soft closure of $h_A$ in $g_A$ is the smallest fuzzy closed set containing $h_A$, so fuzzy soft closure of $h_A$ in $g_A$ is contained in $fsclh_A \overset{\sim}{\bigcap}g_A$.\\
On the other hand, if $w_A$ denotes the fuzzy soft closure of $h_A$ in $g_A$, then $w_A$ is a fuzzy soft closed set in $g_A \Rightarrow ~w_A = k_A \overset{\sim}{\bigcap}g_A$ where $k_A$ is a fuzzy soft closed set in $f_A$(by theorem 2.14). Then $k_A$ is fuzzy soft closed containing $h_A \Rightarrow fsclh_A \overset{\sim}{\subset} k_A \Rightarrow fsclh_A \overset{\sim}{\bigcap} g_A \overset{\sim}{\subset} k_A \overset{\sim}{\bigcap} g_A =w_A$. 
\end{proof}

\section{Fuzzy Soft separation axioms}

Here, we introduce and study various separation axioms for a fuzzy soft topological space. 

\begin{defn}
 A fuzzy soft topological space $(f_A, \tau)$ is said to be a fuzzy soft $T_0-$ space if for every pair of disjoint fuzzy soft points $e_{h_A}, e_{g_B}, \exists$
 a fuzzy soft open set containing one but not the other.
\end{defn}

\begin{exm}
A discrete fuzzy soft topological space is a fuzzy soft $T_0-$ space since every $e_{h_A}$ is a fuzzy soft open set in the discrete space.
\end{exm}

\begin{thm}
 A fuzzy soft subspace of a fuzzy soft $T_0-$ space is fuzzy soft $T_0$.
\end{thm}

\begin{proof}
Let $(g_A, \tau_{g_A})$ be a fuzzy soft subspace of a fuzzy soft $T_0-$ space $(f_A, \tau)$ and let $e_{k1_B}, e_{k2_B}$ be two distinct fuzzy soft points of $g_A$. Then these fuzzy soft points are also in $f_A \Rightarrow ~\exists$ a fuzzy soft open set $h_A$ containing one fuzzy soft point but not the other $\Rightarrow g_A \overset{\sim}{\bigcap} h_A$,where $h_A \in \tau$ is a fuzzy soft open set in $\tau_{g_A}$ containing one fuzzy soft point but not the other.
\end{proof}

\begin{defn}
 A fuzzy soft topological space $(f_A, \tau)$ is said to be a fuzzy soft $T_1-$ space if for distinct pair of fuzzy soft points $e_{g_A},  e_{k_A}$ of  $f_A,~ \exists$ fuzzy soft open sets $s_A$ and $h_A$ such that \\
 $e_{g_A} \overset{\sim}{\in} s_A$ and $e_{g_A} \overset{\sim}{\notin} h_A$;\\
 $e_{k_A} \overset{\sim}{\in} h_A$ and $e_{k_A} \overset{\sim}{\notin} s_A$.
\end{defn}

\begin{thm}
If every fuzzy soft point of a fuzzy soft topological space $(f_A, \tau)$ is fuzzy soft closed then $(f_A, \tau)$ is fuzzy soft $T_1$.
\end{thm}

\begin{proof}
Let $e_{h_A} =\{e_j=\{h^i_{\alpha_i}~|~ i=1, 2, . . . ,n\}\}, e_{k_A}=\{e_m=\{h^i_{\beta_i}~|~ i=1, 2, . . . ,n\}\}$, where $e_j, e_m$ are distinct parameters be distinct fuzzy soft point of $f_A$. 
\begin{enumerate}
\item $\alpha_i, \beta_i \leq 0.5$.\\
Then we can always find some $\gamma_i$ and $\delta_i$ such that $\alpha_i \leq \gamma_i , \beta_i \leq \delta_i \Rightarrow \alpha_i \leq 1-\gamma_i , \beta_i \leq 1-\delta_i \Rightarrow $ the fuzzy soft sets $e_{l_A} =\{e_j=\{h^i_{\gamma_i}~|~ i=1, 2, . . . ,n\}\}$ and $e_{t_A}=\{e_m=\{h^i_{\delta_i}~|~ i=1, 2, . . . ,n\}\}$ are such that their complements are disjoint fuzzy soft open sets containing $e_{h_A}$ and $e_{k_A}$ respectively. 
\item $\alpha_i, \beta_i > 0.5$.\\
Then we can always find some $\gamma_i$ and $\delta_i$ such that $\gamma_i \leq \alpha_i , \delta_i \leq \beta_i \Rightarrow \alpha_i \leq 1-\gamma_i , \beta_i \leq 1-\delta_i \Rightarrow $ the fuzzy soft sets $e_{l_A} =\{e_j=\{h^i_{\gamma_i}~|~ i=1, 2, . . . ,n\}\}$ and $e_{t_A}=\{e_m=\{h^i_{\delta_i}~|~ i=1, 2, . . . ,n\}\}$ are such that their complements are disjoint fuzzy soft open sets containing $e_{h_A}$ and $e_{k_A}$ respectively. 

\end{enumerate}
\end{proof}

\begin{thm}
 A fuzzy soft subspace of a fuzzy soft $T_1-$ space is fuzzy soft $T_1$.
\end{thm}
\begin{defn}
 A fuzzy soft topological space $(f_A, \tau)$ is said to be a fuzzy soft $T_2-$ space if and only if for distinct fuzzy soft points $e_{g_A}, e_{k_A}$ of $f_A, \exists$ disjoint fuzzy soft open sets $h_A$ and $s_A$ such that $e_{g_A}\overset{\sim}{\in} h_A$ and $e_{k_A} \overset{\sim}{\in} s_A$. 
\end{defn}
\begin{thm}
If every fuzzy soft point of a fuzzy soft topological space $(f_A, \tau)$ is fuzzy soft closed then $(f_A, \tau)$ is fuzzy soft $T_2$.
\end{thm}

\begin{thm}
A fuzzy soft subspace of a fuzzy soft $T_2-$ space is fuzzy soft $T_2$.
\end{thm}

\begin{thm}
A fuzzy soft topological space $(f_A, \tau)$ is fuzzy soft $T_2$ if and only if for distinct fuzzy soft points $e_{g_A}, e_{k_A}$ of $f_A, \exists$ a fuzzy soft open set $s_A$ containing $e_{g_A}$ but not $e_{k_A}$ such that $e_{k_A} \overset{\sim}{\notin} fscls_A$
\end{thm}

\begin{proof}
($\Rightarrow$)
Let $(f_A, \tau)$ be fuzzy soft $T_2$ and $e_{g_A}, e_{k_A}$ be distinct fuzzy soft points. So $\exists$ distinct fuzzy soft open sets $h_A$ and $b_A$ such that $e_{k_A} \overset{\sim}{\in} h_A,~ e_{g_A} \overset{\sim}{\in} b_A \Rightarrow e_{g_A} \overset{\sim}{\in} h_A^c$. So $h_A^c$ is a fuzzy soft open set containing $e_{g_A}$ but not  $e_{k_A}$ and $fsclh_A^c = h_A^c$.
($\Leftarrow$)
 Take a pair of distinct fuzzy soft points $e_{g_A}$ and $e_{k_A}$ of $f_A, \exists$ a fuzzy soft open set $s_A$ containing $e_{g_A}$ but not $e_{k_A}$ such that $e_{k_A} \overset{\sim}{\notin} fscls_A \Rightarrow e_{k_A} \overset{\sim}{\in} (fscls_A)^c \Rightarrow s_A $ and $(fscls_A)^c$ are disjoint fuzzy soft open set containing $e_{g_A}$ and $e_{k_A}$ respectively.

\end{proof}

\begin{defn}
 A fuzzy soft topological space $(f_A, \tau)$ is said to be a fuzzy soft regular space if for every fuzzy soft point $e_{h_A}$  and fuzzy soft closed set $k_A$ not containing $e_{h_A},~ \exists$ disjoint fuzzy soft open sets $g_{1A}, g_{2A}$ such that $e_{h_A} \overset{\sim}{\in} g_{1A}$ and $k_A \overset{\sim}{\subseteq} g_{2A}$.

 A fuzzy soft regular $T_1-$ space is called a fuzzy soft $T_3-$ space, 
\end{defn}

\begin{rem}
It can be shown that the property of being fuzzy soft $T_3$ is hereditary. 
\end{rem}

\begin{rem}
Every fuzzy soft $T_3-$ space is fuzzy soft $T_2-$ space, every fuzzy soft $T_2-$ space is fuzzy soft $T_1-$ space and every  fuzzy soft $T_1-$ space is fuzzy soft $T_0-$ space. 
\end{rem}
 
\begin{thm}
 A fuzzy soft topological space $(f_A, \tau)$ in which every fuzzy soft point is fuzzy soft closed, is fuzzy soft regular iff for a fuzzy soft open set $g_A$ containing a fuzzy soft point $e_{h_A}$, there exists a fuzzy soft open set $s_A$ containing $e_{h_A}$ such that $fscls_A \overset{\sim}{\subset} g_A$. 
\end{thm}

\begin{proof}
Take a fuzzy soft open set $g_A$ containing $e_{h_A}$ in a regular fuzzy soft topological space $(f_A, \tau)$. Then $g_A^c$ is fuzzy soft closed. By hypothesis, $\exists$ disjoint fuzzy soft open sets $s_A$ and $w_A$ such that $e_{h_A} \overset{\sim}{\in} s_A $ and $g_A^c \overset{\sim}{\subset} w_A $. Now, $s_A$ and $w_A$ are disjoint, so $s_A \overset{\sim}{\subset} w_A^c \Rightarrow~ fscls_A \overset{\sim}{\subset} w_A^c \Rightarrow~ fscls_A \overset{\sim}{\subset} g_A$.

Conversely, assume the hypothesis. Take a fuzzy soft closed set $k_A$ not containing a fuzzy soft point $e_{h_A} \overset{\sim}{\notin} k_A$. Then $k_A^c$ is a fuzzy soft open set containing the fuzzy $e_{h_A} \Rightarrow ~\exists$ a fuzzy soft open set $s_A$ containing $e_{h_A}$ such that $fscls_A \overset{\sim}{\subset} k_A^c \Rightarrow k_A \overset{\sim}{\subset} (fscls_A)^c \Rightarrow ~(fscls_A)^c$ is a fuzzy soft open set containing $k_A$ and $s_A \overset{\sim}{\bigcap} (fscls_A)^c = \overset{\sim}{\Phi}$.
\end{proof}

\begin{defn}
 A fuzzy soft topological space $(f_A, \tau)$ is said to be a fuzzy soft normal space if for every pair of disjoint fuzzy soft closed sets $h_A$ and $k_A, \exists$ disjoint fuzzy soft open sets $g_{1A}, g_{2A}$ such that\\ 
$h_A \overset{\sim}{\subseteq} g_{1A}$ and $k_A \overset{\sim}{\subseteq} g_{2A}$.

A fuzzy soft normal $T_1-$ space is called a fuzzy soft $T_4-$ space. 
\end{defn}

\begin{rem}
Every fuzzy soft $T_4-$ space is fuzzy soft $T_3$.
\end{rem}

\begin{thm}
A fuzzy soft topological space $(f_A, \tau)$ is fuzzy soft normal iff for any fuzzy soft closed set $h_A$ and fuzzy soft open set $g_A$ containing $h_A$, there exists a fuzzy soft open set $s_A$ such that $h_A \overset{\sim}{\subset} s_A$ and $fscls_A \overset{\sim}{\subset} g_A$. 
\end{thm}

\begin{proof}
Let $(f_A, \tau)$ be fuzzy soft normal space and $h_A$ be a fuzzy soft closed set and $g_A$ be a fuzzy soft open set containing $h_A \Rightarrow h_A $ and $g_A^c$ are disjoint fuzzy soft closed sets $\Rightarrow ~\exists$ disjoint fuzzy soft open sets $g_{1A}, g_{2A}$ such that 
$h_A \overset{\sim}{\subset} g_{1A}$ and $g_A^c\overset{\sim}{\subseteq} g_{2A}$. Now 
$g_{1A} \overset{\sim}{\subset} g_{2A}^c \Rightarrow fsclg_{1A} \overset{\sim}{\subset} fsclg_{2A}^c =g_{2A}^c $ 
Also, $g_A^c \overset{\sim}{\subset} g_{2A} \Rightarrow g_{2A}^c \overset{\sim}{\subset} g_A \Rightarrow fsclg_{1A} \overset{\sim}{\subset}g_A$.

Conversely, let $l_A$ and $k_A$ be any disjoint pair fuzzy soft closed sets $\Rightarrow l_A \overset{\sim}{\subset} k_A^c $, then by hypothesis there exists a fuzzy soft open set $s_A$ such that $l_A \overset{\sim}{\subset} s_A$ and $fscls_A \overset{\sim}{\subset} k_A^c \Rightarrow k_A\overset{\sim}{\subset} (fscls_A)^c \Rightarrow s_A$ and $(fscls_A)^c$ are disjoint fuzzy soft open sets such that $l_A \overset{\sim}{\subset} s_A$ and $k_A \overset{\sim}{\subset} (fscls_A)^c$.
\end{proof}

\begin{thm}
A fuzzy soft closed subspace of a fuzzy soft normal space is fuzzy soft normal.
\end{thm}

\section{Fuzzy soft connectedness}

In this section, we introduce and study fuzzy soft connectedness of fuzzy soft topological spaces.

\begin{defn}
A fuzzy soft separation of a fuzzy soft topological space $(f_A, \tau)$ is a pair $h_A, k_A$ of disjoint non empty fuzzy soft open sets whose union is $f_A$.

If there does not exist a  fuzzy soft separation of $f_A$, then the fuzzy soft topological space is said to be fuzzy soft connected, otherwise fuzzy soft disconnected.
\end{defn}

\begin{exm}
\begin{enumerate}
\item The discrete fuzzy soft topological space with more than one member is always disconnected;
\item The indiscrete fuzzy soft topological space is always connected.
\end{enumerate}
\end{exm}
\begin{thm}
A fuzzy soft topological space $(f_A, \tau)$ is fuzzy soft disconnected $\Leftrightarrow \exists$ a non empty proper fuzzy soft subset of $f_A$ which is both fuzzy soft open and fuzzy soft closed.
\end{thm}
 
\begin{proof}
Let $k_A$ be a non empty proper subset of $f_A$ which is both fuzzy soft open and fuzzy soft closed. Now $h_A=(k_A)^c$ is non empty proper subset of $f_A$ which is also both fuzzy soft open and fuzzy soft closed $\Rightarrow fsclk_A= k_A$ and $fsclh_A = h_A \Rightarrow f_A$ can be expressed as the union of two separated fuzzy soft sets $k_A, h_A$ and so, is fuzzy soft disconnected.

Conversely, let $f_A$ be fuzzy soft disconnected $\Rightarrow ~\exists$ non empty fuzzy soft subsets $k_A$ and $h_A$ such that 
$fsclk_A \overset{\sim}{\bigcap} h_A = \overset{\sim}{\Phi}, k_A \overset{\sim}{\bigcap} fsclh_A = \overset{\sim}{\Phi_A}$ and 
$k_A \overset{\sim}{\bigcup} h_A =f_A$. Now $k_A \overset{\sim}{\subseteq} fsclk_A$ and 
$fsclk_A \overset{\sim}{\bigcap} h_A = \overset{\sim}{\Phi_A} \Rightarrow k_A \overset{\sim}{\bigcap} h_A = \overset{\sim}{\Phi_A} \Rightarrow h_A = (k_A)^c$.
Then $k_A \overset{\sim}{\bigcup} fsclh_A = f_A$ and $k_A \overset{\sim}{\bigcap} fsclh_A = \overset{\sim}{\Phi_A} \Rightarrow k_A = (fsclh_A)^c$ and similarly $h_A = (fsclk_A)^c \Rightarrow k_A, h_A$ are fuzzy soft open sets being the complements of fuzzy soft closed sets. Also $h_A = (k_A)^c \Rightarrow$ they are also fuzzy soft closed.  
 
\end{proof}

\begin{thm}
 If the fuzzy soft sets $h_A$ and $k_A$ form a fuzzy soft separation of $f_A$, and if $(g_A, \tau_{g_A} )$ is a fuzzy soft connected subspace of $f_A$,
 then $g_A \overset{\sim}{\subset} h_A$ or $g_A \overset{\sim}{\subset} k_A$.
\end{thm}

\begin{proof}
 Since $h_A$ and $k_A$ are disjoint fuzzy soft open sets, so are $h_A \overset{\sim}{\bigcap} g_A$ and $k_A \overset{\sim}{\bigcap} g_A$ and their union gives $g_A$, i.e. they would constitute a fuzzy soft separation of $g_A$, a contradiction.
 Hence, one of $h_A \overset{\sim}{\bigcap}g_A$ and $k_A \overset{\sim}{\bigcap} g_A$ is empty and so $g_A$ is entirely contained in on of them.
\end{proof}

\begin{thm}
If $g_A$ is a fuzzy soft subspace of $f_A$, a separation of $g_A$ is a pair of disjoint non empty fuzzy soft sets $k_A$ and $h_A$ whose union is $g_A$, such that $k_A \overset{\sim}{\bigcap} fsclh_A = \overset{\sim}{\Phi}$ and $h_A \overset{\sim}{\bigcap} fsclk_A = \overset{\sim}{\Phi}$
\end{thm}
\begin{proof}
Suppose $k_A$ and $h_A$ forms a separation of $g_A$. Then $k_A$ is both fuzzy soft open and fuzzy soft closed in $g_A$. The fuzzy soft closure of $k_A$ in $g_A$ is $fsclk_A \overset{\sim}{\bigcap}g_A$. Since $k_A$ is fuzzy soft closed in $g_A$, $k_A = fsclk_A \overset{\sim}{\bigcap}g_A \Rightarrow fsclk_A \overset{\sim}{\bigcap}h_A = \overset{\sim}{\Phi}$. By similar argument $fsclh_A \overset{\sim}{\bigcap}k_A = \overset{\sim}{\Phi}$.\\
Conversely, let $k_A$ and $h_A$ are disjoint non empty fuzzy soft sets whose union is $g_A$ such that $k_A \overset{\sim}{\bigcap} fsclh_A = \overset{\sim}{\Phi}$ and $h_A \overset{\sim}{\bigcap} fsclk_A = \overset{\sim}{\Phi} \Rightarrow g_A \overset{\sim}{\bigcap} fsclh_A = \overset{\sim}{\Phi}$ and $g_A \overset{\sim}{\bigcap} fsclk_A = \overset{\sim}{\Phi} \Rightarrow ~h_A$ and $k_A$ are fuzzy soft closed in $g_A$. Also $h_A=k_A^c$ implies both $k_A$ and $h_A$ are fuzzy soft open in $g_A$. 
\end{proof}

\begin{thm}
 Let $g_A$ be a fuzzy soft connected subspace of $f_A$. If $g_A \overset{\sim}{\subset} k_A \overset{\sim}{\subset} fsclg_A$, then $k_A$ is also fuzzy soft connected. 
\end{thm}

\begin{proof}
Let the soft set $k_A$ satisfies the hypothesis.  If possible, let $h_A$ and $s_A$ form a fuzzy soft separation of $k_A$.
Then by theorem 4.2, $g_A \overset{\sim}{\subset} h_A$ or $g_A \overset{\sim}{\subset} s_A$. Let $g_A \overset{\sim}{\subset} h_A  \Rightarrow fsclg_A \overset{\sim}{\subset} fsclh_A$; since $fsclh_A$ and $s_A$ are disjoint, $fsclg_A$ cannot intersect $s_A$.
This contradicts the fact that $s_A$ is a nonempty.  
\end{proof}

\begin{rem}
In particular $fsclg_A$ is fuzzy soft connected if $g_A$ is fuzzy soft connected.
\end{rem}

\begin{rem}
A fuzzy soft topological space is fuzzy soft connected iff $\overset{\sim}{\Phi}$ and $f_A$ are the only sets which are both fuzzy soft open and fuzzy soft closed.
\end{rem}

\begin{thm}
Arbitrary union of fuzzy soft connected subsets of $(f_A, \tau)$ that have non empty intersection is fuzzy soft connected.  
\end{thm}

\begin{proof}
Let $\{(g_{A_{\scriptscriptstyle {\lambda}}}, \tau_{g_{A_{\scriptscriptstyle {\lambda}}}})~|~ \lambda \in \Lambda \}$ be a collection of fuzzy soft connected subspaces of $(f_A, \tau)$ with non empty intersection. If possible, take a fuzzy soft separation  $h_A, ~k_A$ of $g_A = \overset{\sim}{\underset{\lambda \in \Lambda}{\bigcup}} g_{A_{\scriptscriptstyle {\lambda}}}$. Now for each $\lambda, h_A \overset{\sim}{\bigcap} g_{A_{\scriptscriptstyle {\lambda}}}$ and $k_A \overset{\sim}{\bigcap} g_{A_{\scriptscriptstyle {\lambda}}}$ are disjoint fuzzy soft open sets in the subspace such that their union gives $g_{A_{\scriptscriptstyle {\lambda}}}$. As $g_{A_{\scriptscriptstyle {\lambda}}}$ is connected for each $\lambda$, one of $h_A \overset{\sim}{\bigcap} g_{A_{\scriptscriptstyle {\lambda}}}$ and $k_A \overset{\sim}{\bigcap} g_{A_{\scriptscriptstyle {\lambda}}}$ must be empty (by theorem 4.2). Suppose, $h_A \overset{\sim}{\bigcap} g_{A_{\scriptscriptstyle {\lambda}}} = \overset{\sim}{\Phi_A} \Rightarrow k_A \overset{\sim}{\bigcap} g_{A_{\scriptscriptstyle {\lambda}}} = g_{A_{\scriptscriptstyle {\lambda}}} \Rightarrow g_{A_{\scriptscriptstyle {\lambda}}} \overset{\sim}{\subset}k_A~~\forall ~\lambda \in 
\Lambda \Rightarrow \overset{\sim}{\underset{\lambda \in \Lambda}{\bigcup}} g_{A_{\scriptscriptstyle {\lambda}}} \overset{\sim}{\subset}k_A \Rightarrow h_A \overset{\sim}{\bigcup} k_A \overset{\sim}{\subset}k_A \Rightarrow ~h_A$ is empty, a contradiction.
\end{proof}

\begin{thm}
Arbitrary union of a family of fuzzy soft connected subsets of $(f_A, \tau)$ such that one of the members of the family has non empty intersection with every member of the family, is fuzzy soft connected.  
\end{thm}
\begin{proof}
Let $\{(g_{A_{\scriptscriptstyle {\lambda}}}, \tau_{g_{A_{\scriptscriptstyle {\lambda}}}})~|~ \lambda \in \Lambda \}$ be a collection of fuzzy soft connected subspaces of $(f_A, \tau)$ and $g_{A_{\scriptscriptstyle {\lambda_0}}}$ be a fixed member such that $g_{A_{\scriptscriptstyle {\lambda_0}}} \overset{\sim}{\bigcap} g_{A_{\scriptscriptstyle {\lambda}}} \neq \overset{\sim}{\Phi_A}$ for each $\lambda \in \Lambda$. Then by theorem 4.7, 
$h _{A_{\scriptscriptstyle {\lambda}}} = g_{A_{\scriptscriptstyle {\lambda_0}}} \overset{\sim}{\bigcup} g_{A_{\scriptscriptstyle {\lambda}}} $ is a fuzzy soft connected for each $\lambda \in \Lambda$. Now, 

 $\overset{\sim}{\underset{\lambda \in \Lambda}{\bigcup}}h _{A_{\scriptscriptstyle {\lambda}}} = \overset{\sim}{\underset{\lambda \in \Lambda}{\bigcup}} (g_{A_{\scriptscriptstyle {\lambda_0}}} \overset{\sim}{\bigcup} g_{A_{\scriptscriptstyle {\lambda}}})
 = \overset{\sim}{\underset{\lambda \in \Lambda}{\bigcup}} g_{A_{\scriptscriptstyle {\lambda}}}$
\\
and 

$\overset{\sim}{\underset{\lambda \in \Lambda}{\bigcap}}h _{A_{\scriptscriptstyle {\lambda}}} = \overset{\sim}{\underset{\lambda \in \Lambda}{\bigcap}} (g_{A_{\scriptscriptstyle {\lambda_0}}} \overset{\sim}{\bigcup} g_{A_{\scriptscriptstyle {\lambda}}}) =  g_{A_{\scriptscriptstyle {\lambda_0}}}\overset{\sim}{\underset{\lambda \in \Lambda}{\bigcap}} (\overset{\sim}{\bigcup} g_{A_{\scriptscriptstyle {\lambda}}}) \neq \overset{\sim}{\Phi_A}$.

Therefore, by theorem 4.7 $\overset{\sim}{\underset{\lambda \in \Lambda}{\bigcup}}h _{A_{\scriptscriptstyle {\lambda}}}= \overset{\sim}{\underset{\lambda \in \Lambda}{\bigcup}} g_{A_{\scriptscriptstyle {\lambda}}}$ is fuzzy soft connected.
\end{proof}

\begin{thm}
If $(f_A, \tau_2)$ is a fuzzy soft connected space and $\tau_1$ is fuzzy soft coarser than $\tau_2$, then $(f_A, \tau_1)$ is also fuzzy soft connected.
\end{thm}

\begin{proof}
Assume that $k_A, h_A$ form a fuzzy soft separation of $(f_A, \tau_1)$. Now $k_A, h_A \in \tau_1 \Rightarrow k_A, h_A \in \tau_2 \Rightarrow ~k_A, h_A$ form a fuzzy soft separation of $(f_A, \tau_2)$, a contradiction.

\end{proof}

\section{Conclusion} 
This paper investigates properties of separation axioms and connectedness of fuzzy soft topological spaces. Several properties of neighborhood system of a fuzzy soft point are discussed. Other concepts like, compactness and continuity for a fuzzy soft topological space etc can be studied further. \\


\begin{thebibliography}{99}
\bibitem{S6}
P.K. Maji, R. Biswas, A.R. Roy, Fuzzy Soft Sets, J. Fuzzy Math., 9(2001), pp 589-602.


\bibitem{S1}
D. Molodtsov, Soft Set Theory-First Results, Comp. Math. Appl., 37(1999), pp 19-31.
\bibitem{S7}
S. Roy, T. K. Samanta, A note on fuzzy soft topological spaces, Annals of fuzzy mathematics and informatics (Article in press).
\bibitem{S5}
B. Tanay, M. Burc Kandemir, Topological structure of fuzzy soft sets, Comp. Math. Appl., 61(2011), pp 2952-2957. 
\end{thebibliography}
\end{document}